\documentclass[12pt]{amsart}

\usepackage{hyperref}

\usepackage{amsmath,amssymb}

\usepackage{graphicx}
\usepackage[letterpaper,margin=1in]{geometry}

\newtheorem{thm}{Theorem}[section]
\newtheorem{prop}[thm]{Proposition}
\newtheorem{lemma}[thm]{Lemma}
\newtheorem{cor}[thm]{Corollary}

\theoremstyle{definition}

\newtheorem{definition}[thm]{Definition}

\newtheorem{example}[thm]{Example}
\newtheorem{remark}[thm]{Remark}

\newcommand{\PP}{\mathbb{P}}

\DeclareMathOperator{\rank}{rank}

\DeclareMathOperator{\Pf}{Pf}

\newcommand{\defining}[1]{\textbf{#1}}

\title{Decompositions of ideals of minors meeting a submatrix}

\date{June 24, 2014}

\author{Kent M. Neuerburg}
\address{Kent M. Neuerburg \\
Department of Mathematics \\
SLU 10687 \\
Hamond, LA 70401 \\
USA}
\email{kneuerburg@selu.edu}

\author{Zach Teitler}
\address{Zach Teitler \\
Department of Mathematics \\
Boise State University \\
1910 University Drive \\
Boise, ID 83706 \\
USA}
\email{zteitler@boisestate.edu}


\thanks{The authors were partially supported by a grant from the University of Louisiana Board of Regents, LEQSF(2007-10)-RD-A-28.}

\begin{document}

\begin{abstract}
We compute the primary decomposition of certain ideals generated by
subsets of minors in a generic matrix or in a generic symmetric matrix,
or subsets of Pfaffians in a generic skew-symmetric matrix.
Specifically, the ideals we consider are generated by minors
that have at least some given number of rows and columns in certain submatrices.
\end{abstract}

\maketitle

\section{Introduction}

The paper \cite{MR853243} concerns ideals of minors fixing a submatrix,
meaning the set of minors in an $m \times n$ matrix that involve all $r$ of the first $r$ columns of the matrix.
One of the main results of that paper, Theorem~A, gives the primary decomposition of the ideal
generated by this set of minors.
We generalize this to consider minors that involve at least $r$ of the first $a$ columns:
\begin{thm}\label{thm: two block minors}
Let $k$ be a field, let $X$ be a \emph{generic} $m \times n$ matrix,
that is $X = (x_{i,j})$ for $1 \leq i \leq m$, $1 \leq j \leq n$,
and let $R = k[X] = k[x_{i,j}]$.
Regard $X$ as a block matrix, $X = (AB)$, where $A$ has size $m \times a$ and $B$ has size $m \times (n-a)$.
Let $J$ be the ideal generated by the set of $t$-minors of $X$ that involve at least $r$ columns of $A$,
let $I_t(X)$ be the ideal generated by the $t$-minors of $X$,
and similarly let $I_r(A)$ be the ideal generated by the $r$-minors of $A$.
Then $J = I_t(X) \cap I_r(A)$.
\end{thm}
We generalize further than this, to allow several blocks
as well as restrictions on both rows and columns.
We also give similar statements for ideals generated by
sets of minors of a generic symmetric matrix,
requiring some number of rows or columns in certain submatrices.
Before we give these statements,
we consider one possible application in the setting of the two-block theorem above.

It is sometimes useful to consider, for a homogeneous ideal $I$,
the ideal $I_{\leq d}$ generated by the forms in $I$ of degree $\leq d$.
For example, in resolving the singularities of the affine cone $V(I) \subset \mathbb{A}^n$,
upon blowing up the origin, the total transform of $I$ may have embedded components
supported along the projective variety $V(I_{\leq d})$, lying in the exceptional divisor $\cong \mathbb{P}^{n-1}$,
for various $d$; see \cite{zct:mla}.
When $I$ is a determinantal ideal, generated by minors of a matrix whose entries are homogeneous forms,
then $I_{\leq d}$ is generated by just some of the minors of the matrix.

\begin{cor}
Let $X = (x_{i,j})$ be a generic $m \times n$ matrix,
regarded as consisting of two blocks, $X = (AB)$,
where $A$ has size $m \times a$ and $B$ has size $m \times (n-a)$.
Fix the ring $R = k[X] = k[x_{i,j}]$ 
where every entry in $A$ has degree $p$ and every entry in $B$ has degree $q > p$;
that is, $\deg(x_{i,j}) = p$ if $1 \leq j \leq a$, $\deg(x_{i,j}) = q$ otherwise.
Fix $t$ and $d$.
Let $I_t(X)_{\leq d}$ be the ideal generated by those $t$-minors of $X$
of degree less than or equal to $d$,
and for each $r$ let $I_r(A)$ be the ideal generated by the $r$-minors of $A$.
Then $I_t(X)_{\leq d} = I_t(X) \cap I_r(A)$ for $r = \lceil \frac{tq-d}{q-p} \rceil$.
\end{cor}
Indeed, a $t \times t$ minor $M$ with $r$ columns in $A$ and $t-r$ columns in $B$ will have degree
$\deg M = pr+q(t-r)$; the value of $r$ in the statement is the least integral solution to $\deg M \leq d$.

\begin{cor}
Consider the vector bundles
$F = \mathcal{O}_{\PP^N}^m$, $G = \mathcal{O}_{\PP^N}(p)^{\oplus a} \oplus \mathcal{O}_{\PP^N}(q)^{\oplus n-a}$.
Let $f : F \to G$ be a general map,
let $f' : F \to \mathcal{O}_{\PP^N}(p)^{\oplus a}$ be the induced map,
let $\Delta_t(f)$ be the degeneracy locus, $\Delta_t(f) = \{ x \in \PP^N \mid \rank f_x < t \}$,
and let $\Delta_t(f)_{\leq d}$ be the locus defined by the ideal $I(\Delta_t(f))_{\leq d}$.
Then $\Delta_t(f)_{\leq d} = \Delta_t(f) \cup \Delta_r(f')$
for $r = \lceil \frac{tq-d}{q-p} \rceil$.
\end{cor}

This is similar to \cite{zct:smoothness-of-envelopes},
which dealt with $\PP^2$ and had $n = m+1$ in order to obtain general Hilbert-Burch matrices of a given type.
(In particular \cite[Prop.~3.4]{zct:smoothness-of-envelopes} simply recreated
a special case of \cite[Thm.~A]{MR853243}.)

Sections~\ref{section: orders and straightening} and \ref{section: order ideals}
review some background of posets, dosets,
algebras with straightening law, and doset algebras with straightening law.
Then we give our results for minors in generic matrices (Section~\ref{section: minors}),
minors in generic symmetric matrices (Section~\ref{section: symmetric minors}),
and Pfaffians in generic skew-symmetric matrices (Section~\ref{section: pfaffians}).

Throughout, all rings are commutative with unity.


\section{Orders and straightening}\label{section: orders and straightening}

A \defining{poset} (partially ordered set)
is a set together with a transitive, reflexive, antisymmetric relation $\leq$.

\begin{definition}[{\cite[Definition~1.0.3]{MR2298637}}]
A \defining{doset} of a poset $P$ is a subset $D \subset P \times P$ such that
\begin{enumerate}
\item $(a,a) \in D$ for all $a \in P$,
\item if $(a,b) \in D$ then $a \leq b$, and
\item if $a \leq b \leq c \in P$, then $(a,c) \in D$ if and only if $(a,b) \in D$ and $(b,c) \in D$.
\end{enumerate}
\end{definition}

\begin{example}
\begin{enumerate}
\item Let $[n] = \{1,\dotsc,n\}$
and let $P_n = 2^{[n]}$, the power set of $[n]$.
We order $P_n$ as follows.
For $A = \{a_1 < \dotsb < a_s\} \subset [n]$
and $B = \{b_1 < \dotsb < b_t\} \subset [n]$,
$A \leq B$ if and only if $s \geq t$ and $a_i \leq b_i$ for $i = 1,\dotsc,t$.
This makes $P_n$ a poset.

Note, $A \leq B$ if and only if in the diagram
\[
\begin{array}{|c|c|c|c|c|c|}
  \hline
  a_1 & a_2 & \cdots & a_t & \cdots & a_s \\
  \hline
  b_1 & b_2 & \cdots & b_t \\
  \cline{1-4}
\end{array}
\]
the first row is at least as long as the second
and the entries are weakly increasing down each column.

\item Fix $m$ and $n$.
Let $P_{m,n} \subset P_m \times P_n$
consist of pairs of subsets $(A,B)$ such that $|A|=|B|$,
with $(A,B) \leq (A',B')$ if and only if $A \leq A'$ and $B \leq B'$.

\item Let $D_n \subset P_n \times P_n$
consist of pairs $(A,B)$ such that $|A|=|B|$ and $A \leq B$.
Then $D_n$ is a doset.
\end{enumerate}
\end{example}

\begin{example}\label{example: posets of minors}
Here are some key examples of posets and dosets of minors in matrices.
\begin{enumerate}
\item $P_{m,n}$ is the \defining{poset of minors} (of an $m \times n$ matrix):
$(A,B) \in P_{m,n}$ corresponds to the minor with rows indexed by $A$ and columns indexed by $B$.
This poset is usually denoted $\Delta(X)$, where $X$ is an $m \times n$ matrix.

\item $D_n$ is the \defining{doset of minors} of a symmetric $n \times n$ matrix:
$(A,B) \in D_n$ corresponds to the minor with rows indexed by $A$ and columns indexed by $B$.
An element of $D_n$ is called a \defining{doset minor}.
We denote this poset $\Delta^s(Y)$, where $Y$ is a symmetric $n \times n$ matrix.
(It is denoted $\Delta(Y)$ in \cite{MR2298637};
we adjoin the $s$ for ``symmetric'' in order to avoid ambiguity.)

The condition $A \leq B$ means that a minor is a doset minor if and only if
the main diagonal of the minor lies in the upper triangle of the matrix (including the diagonal).

\item Let $P_n(2)$ be the subset of $A \in P_n$ such that $|A|$ is even.
Then $P_n(2)$ is the \defining{poset of Pfaffians} of a skew-symmetric $n \times n$ matrix:
$A \in P_n(2)$ corresponds to the Pfaffian of the submatrix with rows and columns indexed by $A$.
Following \cite{MR2298637} we denote this poset $\Pi(Z)$, where $Z$ is a skew-symmetric $n \times n$ matrix.
\end{enumerate}
\end{example}

\begin{definition}[{\cite[\textsection4.A]{MR953963},\cite[Definition~1.0.1]{MR2298637}}]
Let $A$ be a $B$-algebra and $P \subset A$ a subset with a partial order $\leq$.
Then $A$ is a \defining{graded algebra with straightening law} (abbreviated ASL) on $P$ over $B$ if
\begin{enumerate}
\item $A = \bigoplus_{i \geq 0} A_i$ is a graded $B$-algebra such that $A_0 = B$,
$P$ consists of homogeneous elements of positive degree, and $P$ generates $A$ as a $B$-algebra.
\item $A$ is a free $B$-module with a basis given by products
$\xi_1 \dotsm \xi_m$, $m \geq 0$, $\xi_i \in P$, such that $\xi_1 \leq \dotsb \leq \xi_m$.
These products are called \defining{standard monomials}.
\item For all incomparable $\xi, \nu \in P$, the product $\xi \nu$ can be written as a combination
of standard monomials
\[
  \xi \nu = \sum a_\mu \mu, \qquad a_\mu \in B, a_\mu \neq 0, \qquad \text{$\mu$ standard monomial},
\]
in which every $\mu$ contains a factor $\zeta \in P$ such that $\zeta \leq \xi$ and $\zeta \leq \nu$.
These are called \defining{straightening relations}.
\end{enumerate}
\end{definition}

\begin{definition}[{\cite[Definition~1.0.4]{MR2298637}}]
Let $A$ be a $B$-algebra and $D \subset A$ a subset such that $D$ is a doset of a poset $P$.
Then $A$ is a \defining{graded doset algebra with straightening law} (abbreviated DASL)
on $D$ over $B$ if
\begin{enumerate}
\item $A = \bigoplus_{i \geq 0} A_i$ is a graded $B$-algebra such that $A_0 = B$,
$D$ consists of homogeneous elements of positive degree, and $D$ generates $A$ as a $B$-algebra.
\item $A$ is a free $B$-module with a basis given by products
$(\alpha_1,\alpha_2)\dotsm(\alpha_{2k-1},\alpha_{2k})$, $k \geq 1$, $(\alpha_{2i-1},\alpha_{2i}) \in D$,
$\alpha_1 \leq \dotsb \leq \alpha_{2k}$.
These products are called \defining{standard monomials}.
\item Suppose $M = (\alpha_1,\alpha_2)\dotsm(\alpha_{2k-1},\alpha_{2k})$,
with \defining{standard representation} $M = \sum \lambda_N N$, $0 \neq \lambda_N \in B$,
each $N$ a standard monomial.
Let $N = (\beta_1,\beta_2)\dotsm(\beta_{2\ell-1},\beta_{2\ell})$ be one of the standard monomials
appearing in the standard representation of $M$.
Then for every permutation $\sigma$ of $\{1,\dotsc,2k\}$,
the sequence $\{\alpha_{\sigma(1)},\dotsc,\alpha_{\sigma(2k)}\}$
is lexicographically greater than or equal to the sequence $(\beta_1,\dotsc,\beta_{2\ell})$.
\item In the notation above, if there is a permutation $\sigma$ such that
$\alpha_{\sigma(1)} \leq \dotsb \leq \alpha_{\sigma(2k)}$
then the standard monomial $(\alpha_{\sigma(1)},\alpha_{\sigma(2)})\dotsm(\alpha_{\sigma(2k-1)},\alpha_{\sigma(2k)})$
must appear in the standard representation of $M$ with coefficient $\pm1$.
\end{enumerate}
\end{definition}

\begin{example}
Fix an arbitrary commutative ring $B$ with unity.
\begin{enumerate}
\item Let $X$ be an $m \times n$ \defining{generic matrix}, that is $X = (x_{i,j})$,
$1 \leq i \leq m$, $1 \leq j \leq n$, the $x_{i,j}$ variables over $B$.
Then $A = B[X] = B[x_{i,j}]$ is a graded ASL on $\Delta(X)$ over $B$
\cite[Chap.~4]{MR953963},\cite[Thm.~1.0.5]{MR2298637}.

\item Let $Y$ be an $n \times n$ \defining{generic symmetric matrix},
that is $Y = (y_{i,j})$, $1 \leq i, j \leq n$, $y_{i,j} = y_{j,i}$.
Then $A = B[Y]$ is a graded DASL on $\Delta^s(Y)$ over $B$
\cite[Thm.~1.0.10]{MR2298637}.

\item Let $Z$ be an $n \times n$ \defining{generic skew-symmetric matrix},
that is $Z = (z_{i,j})$, $1 \leq i,j \leq n$, $z_{i,j} = -z_{j,i}$, $z_{i,i} = 0$.
Then $A = B[Z]$ is a graded ASL on $\Pi(Z)$ over $B$
\cite[Thm.~1.0.14]{MR2298637}.
\end{enumerate}
\end{example}

\section{Order ideals}\label{section: order ideals}
We use ASLs and DASLs entirely for the following properties.

\begin{definition}
Let $P$ be a poset.
An \defining{order ideal} is a subset $I \subset P$
such that if $\alpha \in I$ and $\beta \leq \alpha$
then $\beta \in I$.
The order ideal \defining{generated by} $S \subset P$ is the smallest order ideal containing $S$,
that is, $\{ \alpha \in P \mid \alpha \leq s \text{ for some $s \in S$} \}$.
The order ideal \defining{cogenerated by} $S \subset P$ is the largest order ideal disjoint from $S$,
that is, $\{ \alpha \in P \mid \alpha \not\leq s \text{ for all $s \in S$} \}$.
\end{definition}

When $A$ is an ASL on $P$ and $I \subset P$, we write $AI$ for the (ring) ideal generated by $I$.

\begin{lemma}[{\cite[Prop.~5.2]{MR953963}}]\label{lemma: ASL order ideal intersection}
Let $A$ be an ASL on $P$ and let $I, J \subset P$ be ideals.
Then $AI \cap AJ = A(I \cap J)$.
\end{lemma}

We will prove a similar lemma for DASLs.
First, we introduce a partial order for dosets.

\begin{definition}
Let $D$ be a doset of $P$.
Then $D$ is a poset with the partial order $(a,b) \leq_1 (c,d)$ if and only if $a \leq c$ in $P$.
%
A \defining{doset order ideal} is an order ideal in the poset $(D,\leq_1)$.
As before, the ideal generated by $S \subset D$ is the smallest ideal containing $S$
and the ideal cogenerated by $S \subset D$ is the largest ideal disjoint from $S$.
\end{definition}

A DASL on $D$ is not necessarily an ASL on $(D,\leq_1)$.
Again when $A$ is a DASL on $D$ and $I \subset D$ is a doset order ideal,
we write $AI$ for the ring ideal generated by $I$.

\begin{lemma}
Let $A$ be a DASL on $D$ over $B$ and let $I \subset D$ be a doset order ideal.
Then $AI$ is spanned over $B$ by the standard monomials
$N = (\beta_1,\beta_2)\dotsm(\beta_{2\ell-1},\beta_{2\ell})$
such that $(\beta_1,\beta_2) \in I$.
\end{lemma}
\begin{proof}
Let $(\alpha_1,\alpha_2) \in I$, $f \in A$, and let
$N = (\beta_1,\beta_2)\dotsm(\beta_{2\ell-1},\beta_{2\ell})$
be one of the standard monomials appearing in the standard representation of $(\alpha_1,\alpha_2) f$.
The sequence $(\beta_1,\beta_2,\dotsc,\beta_{2\ell})$ is lexicographically less than or equal
to $(\alpha_1,\alpha_2)$,
so in particular $\beta_1 \leq \alpha_1$.
Hence $(\beta_1,\beta_2) \leq_1 (\alpha_1,\alpha_2)$ and hence $(\beta_1,\beta_2) \in I$.
Thus every standard monomial appearing in every element of $AI$ has a factor in $I$.
\end{proof}

\begin{lemma}\label{lemma: DASL order ideal intersection}
Let $A$ be a DASL on $D$ and let $I, J \subset D$ be ideals.
Then $AI \cap AJ = A(I \cap J)$.
\end{lemma}
\begin{proof}
A standard monomial $N = (\beta_1,\beta_2)\dotsm(\beta_{2\ell-1},\beta_{2\ell})$
appearing in the standard representation of an element of $AI \cap AJ$
has $(\beta_1,\beta_2) \in I$ and $\in J$, hence in $I \cap J$.
This shows $AI \cap AJ \subset A(I \cap J)$ and the reverse inclusion is obvious.
\end{proof}

Finally we recall the following results.

\begin{prop}[{\cite[Thm.~6.3]{MR953963}}]
Let $B$ be a domain, $X$ a generic matrix, $A = B[X]$, and $\delta \in \Delta(X)$,
the poset of minors (see Example~\ref{example: posets of minors}(1)).
Let $I(X,\delta)$ be the ideal in $A$ generated by the order ideal cogenerated by $\delta$.
Then $I(X,\delta)$ is a prime ideal.
\end{prop}

\begin{prop}[{\cite[Theorem 1]{MR0352082}, \cite[Remark~2.5(a)]{MR1279266}}]
\label{prop: doset 1-cogenerated prime}
Let $B$ be a domain, $Y$ a generic symmetric matrix, $A = B[Y]$,
and $\delta \in \Delta^s(Y)$,
the doset of minors (see Example~\ref{example: posets of minors}(2)).
Let $I(Y,\delta)$ be the ideal in $A$ generated by the doset order ideal cogenerated by $\delta$.
Then $I(Y,\delta)$ is a prime ideal.
\end{prop}

\begin{prop}[{\cite[Thm.~2.1.12]{MR2298637}}]
Let	$B$ be a domain, $Z$ a generic skew-symmetric matrix, $A = B[Z]$,
and $\delta \in \Pi(Z)$,
the poset of Pfaffians (see Example~\ref{example: posets of minors}(3)).
Let $I(Z,\delta)$ be the ideal in $A$ generated by the order ideal cogenerated by $\delta$.
Then $I(Z,\delta)$ is a prime ideal.
\end{prop}

\section{Minors}\label{section: minors}

We are interested in ideals generated by certain sets of $t$-minors in a generic matrix $X$.
Specifically, we will require the generating minors to have at least $r_1$ rows in the first $R_1$ rows of $X$,
at least $r_2$ rows contained in the first $R_2$ rows of $X$, and so on;
and similarly for columns.

Let $X$ be a generic $m \times n$ matrix, $X = (x_{i,j})$ for $1 \leq i \leq m$, $1 \leq j \leq n$,
and fix $A = B[X] = B[\{x_{i,j}\}]$ for a commutative ring $B$ with unity.
For $1 \leq t \leq \min(m,n)$, a $t$-minor may be specified by listing its rows and columns;
we write $[a_1,\dotsc,a_t \mid b_1,\dotsc,b_t]$,
where $1 \leq a_1 < \dotsb < a_t \leq m$ and $1 \leq b_1 < \dotsb < b_t \leq n$,
for the minor with rows $a_1,\dotsc,a_t$ and columns $b_1,\dotsc,b_t$.

Fix sequences $1 \leq R_1 \leq \dotsb \leq R_p \leq m$
and $1 \leq C_1 \leq \dotsb \leq C_q \leq n$
where $p, q \geq 0$.
The sequences $R = (R_1,\dotsc,R_p)$ and $C = (C_1,\dotsc,C_q)$
(possibly empty if $p=0$ or $q=0$)
describe the division of $X$ into row and column blocks, respectively.
Specifically, let $X_{R_i}$ be the submatrix of $X$ consisting of the first $R_i$ rows
and let $X^{C_j}$ be the submatrix consisting of the first $C_j$ columns.
Fix also sequences $r = (r_1,\dotsc,r_p)$ and $c = (c_1,\dotsc,c_q)$.

We are interested in the $t$-minors that have at least $r_i$ rows contained in $X_{R_i}$
and at least $c_j$ columns contained in $X^{C_j}$, for each $i, j$
(with no restriction if $p=0$ or $q=0$).
\begin{thm}\label{thm: minors}
Let $B$ be a ring and $A = B[X]$.
Let $J = J(X,t,R,C,r,c)$ be the ideal generated by $t$-minors of $X$
that have at least $r_i$ rows contained in $X_{R_i}$ for each $1 \leq i \leq p$ (no restriction if $p=0$)
and at least $c_j$ columns contained in $X^{C_j}$ for each $1 \leq j \leq q$ (no restriction if $q=0$).
Then
\begin{equation}\label{eq: minor decomposition}
  J = I_t(X) \cap I_{r_1}(X_{R_1}) \cap \dotsb \cap I_{r_p}(X_{R_p})
             \cap I_{c_1}(X^{C_1}) \cap \dotsb \cap I_{c_q}(X^{C_q}) .
\end{equation}
\end{thm}

\begin{example}
When $p=q=0$, $J = I_t(X)$.

When $p=0$ and $q=1$, we are in the two-block setting of the Introduction.
If also $c_1=C_1$, we recover \cite[Thm.~A]{MR853243}.
\end{example}

\begin{remark}
We are essentially working with the special case of Mohammadi's block adjacent simplicial complexes \cite{Mohammadi:2012fk}
in which each block is contained in the previous one and they all have the last column of the matrix
as a common endpoint (in Mohammadi's indexing; for us, we take blocks to start at the first column or row).
Unlike Mohammadi, we allow non-maximal minors,
we allow restrictions on both the rows and columns appearing in the minor,
and we allow the overlaps between ``consecutive'' blocks to be arbitrarily large.
%
\end{remark}

\begin{proof}
Each of the following sets of minors is an order ideal in $\Delta(X)$:
\begin{enumerate}
\item The set of minors of size $\geq t$, the generating set of $I_t(X)$, is the order ideal generated
by $[m-t+1,\dotsc,m \mid n-t+1,\dotsc,n]$,
or cogenerated by $[1,\dotsc,t-1 \mid 1,\dotsc,t-1]$.
\item The set of $(\geq r_i)$-minors of $X_{R_i}$ is the order ideal generated by
$[R_i-r_i+1,\dotsc,R_i \mid n-r_i+1,\dotsc,n]$,
or cogenerated by $[1,\dotsc,r_i-1,R_i+1,\dotsc,n \mid 1,\dotsc,n-R_i+r_i-1]$.
\item Similarly, the set of $(\geq c_j)$-minors of $X^{C_j}$ is the order ideal generated by
$[m-c_j+1,\dotsc,m \mid C_j-c_j+1,\dotsc,C_j]$,
or cogenerated by $[1,\dotsc,n-C_j+c_j-1 \mid 1,\dotsc,c_j-1,C_j+1,\dotsc,n]$.
\end{enumerate}
By Lemma~\ref{lemma: ASL order ideal intersection},
the intersection of the ideals generated by these sets
is equal to the ideal generated by the intersection of the sets.
\end{proof}

Note, if $B$ is a domain this gives $J$ as an intersection of prime ideals.
However it may fail to be a primary decomposition of $J$,
as redundancies may arise in the following ways.
For example, if $r_j > R_j - R_i$ then every minor containing
at least $r_j$ rows of $X_{R_j}$ must contain at least $r_j - (R_j-R_i)$ rows of $X_{R_i}$;
now if $r_i \leq r_j - R_j + R_i$ then the condition imposed by $r_i$ is implied by the $r_j$ condition
and the prime ideal $I_{r_i}(X_{R_i})$ is redundant.
Or if $t - (m-R_i) \geq r_i$ then every $t$-minor has at least $r_i$ rows in $X_{R_i}$.
Finally there are a few trivial situations: if $r_i > R_i$ the whole thing is zero;
if $R_i = R_j$ or $r_i = r_j$ then one condition is obviously redundant.
These are the only possible redundancies as the following proposition shows.
\begin{prop}
Suppose
\begin{enumerate}
\item $R_1 < \dotsb < R_p$ and $C_1 < \dotsb < C_q$,
\item $r_1 < \dotsb < r_p < t$ and $c_1 < \dotsb < c_q < t$,
\item $0 \leq r_i \leq R_i$ for each $i$ and $0 \leq c_j \leq C_j$ for each $j$,
\item $R_1 - r_1 < \dotsb < R_p - r_p < m-t$ and $C_1 - c_1 < \dotsb < C_q - c_q < n-t$.
\end{enumerate}
Then the intersection \eqref{eq: minor decomposition} is irredundant.
\end{prop}
\begin{proof}
First, fix $1 \leq i \leq p$.
Consider the $t$-minor
\[
  m = [1,\dotsc,r_i-1,R_i+1,\dotsc,t+R_i-r_i+1 \mid 1,\dotsc,t] .
\]
We use $R_i - r_i < m-t$ to verify $t+R_i-r_i+1 \leq m$, so this is a permissible $t$-minor in an $m \times n$ matrix.
For each $j < i$, $m$ has exactly $\min(r_i-1,R_j)$ rows in $X_{R_j}$, and this is $\geq r_j$,
so $m \in I_{r_j}(X_{R_j})$.
For each $j > i$, the number of rows of $m$ in $X_{R_j}$ is either $t$, if $t+R_i-r_i+1 \leq R_j$,
or else $R_j-R_i+r_i-1$, if $R_i+1 \leq R_j \leq t+R_i-r_i+1$.
In the first case $t \geq r_j$ and in the second case $R_j - r_j > R_i - r_i$, so $R_j - R_i + r_i - 1 \geq r_j$;
therefore $m \in I_{r_j}(X_{R_j})$.
And clearly $m$ has only $r_i-1$ rows in $X_{R_i}$.
This shows that
\[
  m \in I_t(X) \cap I_{r_1}(X_{R_1}) \cap \dotsb \cap I_{r_{i-1}}(X_{R_{i-1}})
    \cap I_{r_{i+1}}(X_{R_{i+1}}) \cap \dotsb \cap I_{r_p}(X_{R_p})
\]
but $m \notin I_{r_i}(X_{R_i})$.
Clearly $m \in \bigcap I_{c_j}(X^{C_j})$.
So the term $I_{r_i}(X_{R_i})$ is irredundant for each $i$.
The same argument shows that each $I_{c_j}(X^{C_j})$ is irredundant.

Finally consider the $(t-1)$-minor
\[
  m' = [1,\dotsc,t-1 \mid 1,\dotsc,t-1] .
\]
Since each $r_i < t$, $m'$ has at least $r_i$ rows in each $X_{R_i}$ and similarly at least $c_j$ columns in each $X^{C_j}$.
This shows that the term $I_t(X)$ is irredundant.
\end{proof}

\section{Minors of symmetric matrices}\label{section: symmetric minors}

Now let $Y = (y_{i,j})$ be a generic symmetric $n \times n$ matrix, $y_{i,j} = y_{j,i}$.
Fix sequences $R = (R_1,\dotsc,R_p)$ with $1 \leq R_1 \leq \dotsb \leq R_p \leq n$
and $r = (r_1,\dotsc,r_p)$.
Let $Y_{R_i}$ be the submatrix consisting of the first $R_i$ rows of $Y$.
We are interested in the $t$-minors that have at least $r_i$ rows in $Y_{R_i}$ for each $i$.
Note, at this point we allow all minors, not only doset minors.

\begin{thm}\label{thm: symmetric minors}
Let $B$ be a ring and $A = B[Y]$.
Let $J = J(Y,t,R,r)$ be the ideal generated by $t$-minors of $Y$ that have
at least $r_i$ rows contained in $Y_{R_i}$ for each $1 \leq i \leq p$.
Then
\begin{equation}\label{eq: symmetric minor decomposition}
  J = I_t(Y) \cap I_{r_1}(Y_{R_1}) \cap \dotsb \cap I_{r_p}(Y_{R_p}) .
\end{equation}
If $B$ is a domain then each $I_{r_i}(Y_{R_i})$ is a prime ideal.
\end{thm}
\begin{proof}
First, by \cite[Lemma~2.3]{MR1279266}, every $t$-minor $[a \mid b]$
is a linear combination of doset $t$-minors $[c \mid d]$ with $c \leq a$.
Thus we can take $J$ to be generated by the doset $t$-minors meeting the row conditions.

Next, each of the following sets is a doset order ideal in $\Delta^s(Y)$:
\begin{enumerate}
\item The set of doset minors of size $\geq t$ is the doset order ideal
generated by $[n-t+1,\dotsc,n \mid n-t+1,\dotsc,n]$.

\item The set of doset $(\geq r_i)$-minors of $Y_{R_i}$ is the doset order ideal
generated by $[R_i-r_i+1,\dotsc,R_i \mid n-r_i+1,\dotsc,n]$.
If $[a \mid b] \leq_1 [R_i-r_i+1,\dotsc,R_i \mid n-r_i+1,\dotsc,n]$
then $[a \mid b]$ involves at least $r_i$ rows of $Y_{R_i}$;
by Laplace expansion and \cite[Lemma~2.3]{MR1279266},
$[a \mid b]$ is a linear combination of doset $r_i$-minors of $Y_{R_i}$.
\end{enumerate}
This shows that $J$ is the indicated intersection.

The set of doset $(\geq t)$-minors of $Y$ is
cogenerated by $[1,\dotsc,t-1 \mid 1,\dotsc,t-1]$.
The set of doset $(\geq r_i)$-minors of $Y_{R_i}$
is cogenerated by $m = [1,\dotsc,r_i-1,R_i+1,\dotsc,n \mid 1,\dotsc,r_i-1,R_i+1,\dotsc,n]$
\cite[Remark~2.5(c)]{MR1279266}.
Indeed, $[a \mid b] \not \geq_1 m$ if and only if
$a \not \geq (1,\dotsc,r_i-1,R_i+1,\dotsc,n)$,
if and only if $|a| \geq r_i$ and $a_{r_i} \leq R_i$;
so $[a \mid b]$ involves at most $r_i$ rows of $Y_{R_i}$.
This shows that each of the ideals being intersected is cogenerated by a single doset element.
Therefore if $B$ is a domain then each of them is a prime ideal by Proposition~\ref{prop: doset 1-cogenerated prime}.
\end{proof}

If $B$ is a domain then once again this writes $J$ as an intersection of prime ideals,
but as before it may fail to be a primary decomposition because of redundancy.

\begin{prop}
Suppose
\begin{enumerate}
\item $R_1 < \dotsb < R_p$,
\item $r_1 < \dotsb < r_p < t$,
\item $0 \leq r_i \leq R_i$ for each $i$,
\item $R_1 - r_1 < \dotsb < R_p - r_p < n-t$.
\end{enumerate}
Then the intersection \eqref{eq: symmetric minor decomposition} is irredundant.
\end{prop}
The proof is the same as before.

%

\section{Pfaffians}\label{section: pfaffians}

Let $Z = (z_{i,j})$ be an $n \times n$ generic skew-symmetric matrix,
so that $z_{i,j} = -z_{j,i}$ and $z_{i,i}=0$,
and let $A = B[Z] = B[\{z_{i,j}\}]$.
The Pfaffian of $Z$, denoted $\Pf(Z)$, is a certain polynomial in the entries of $Z$,
with the property that $\Pf(Z)^2 = \det(Z)$.
When $n$ is odd, $\Pf(Z) = \det(Z) = 0$; for $n = 2,4$ we have
\[
  \Pf\begin{pmatrix} 0 & a \\ -a & 0\end{pmatrix} = a,
  \qquad
  \Pf\begin{pmatrix} 0 & a & b & c \\ -a & 0 & d & e \\ -b & -d & 0 & f \\ -c & -e & -f & 0\end{pmatrix} = af-be+cd .
\]
In general, for $n$ even,
\[
  \Pf(Z) = \sum \operatorname{sgn}(\sigma) z_{\sigma(1),\sigma(2)} \dotsm z_{\sigma(n-1),\sigma(n)} ,
\]
where the sum is over all permutations $\sigma \in S_n$ such that $\sigma(2i-1) < \sigma(2i)$ for all $i$
and $\sigma(1) < \sigma(3) < \dotsb < \sigma(2n-1)$.
Equivalently, the sum is over all unordered partitions of $\{1,\dotsc,2n\}$ into pairs;
the restrictions on $\sigma$ simply amount to choosing one representative ordering for each partition.
There is a Laplace-like expansion: for each $j$, $1 \leq j \leq n$,
\[
  \Pf(Z) = \sum_{i<j} (-1)^{i+j+1} z_{i,j} \Pf(Z^{i,j}) + \sum_{i > j} (-1)^{i+j} z_{i,j} \Pf(Z^{i,j}) ,
\]
where $Z^{i,j}$ is the matrix obtained by deleting the $i$th and $j$th rows and columns of $Z$.
See \cite{MR0453723,MR0257105,MR554859}.

A $t$-Pfaffian of $Z$ is given by a list of $t$ rows and the same columns;
we write briefly $[a_1,\dotsc,a_t]$, where $1 \leq a_1 < \dotsb < a_t \leq n$,
for the Pfaffian of the skew-symmetric submatrix
given by the rows $a_1,\dotsc,a_t$ and the same columns.
Of course this is zero if $t$ is odd.

The ideal generated by the size $t$ Pfaffians of $Z$ is denoted $P_t(Z)$.
If $n$ is odd then $P_{n-1}(Z)$ is a prime ideal of height $3$.
More generally, $P_{2p}(Z)$ is a prime ideal of height $\mu(p,n) = (n-2p+1)(n-2p+2)/2$, see \cite{MR554859}.

We are interested in the ideal generated by the subset of Pfaffians with at least $r_1$ rows
in the first $R_1$ rows of $Z$, at least $r_2$ rows in the first $R_2$ rows of $Z$, and so on;
the row condition implies that these Pfaffians meet the corresponding column conditions as well,
i.e., at least $r_1$ columns in the first $R_1$ columns of $Z$, and so on.

Fix a sequence $1 \leq R_1 \leq \dotsb \leq R_p \leq n$, $R = (R_1,\dotsc,R_p)$,
and another sequence $r = (r_1,\dotsc,r_p)$ of the same length.
Let $Z_{R_i}$ be the submatrix of $Z$ consisting of the first $R_i$ rows
and let $Z_{R_i}^{R_i}$ be the $R_i \times R_i$ submatrix of $Z$ in the upper left corner,
consisting of the first $R_i$ rows and the first $R_i$ columns.
We will also need, for each $R_i+1 \leq k \leq n$,
the $(R_i+1) \times (R_i+1)$ submatrix given by the
first $R_i$ rows and columns plus the $k$th row and column,
that is, the set of rows (and columns) corresponding to the set $\{1,\dotsc,R_i,k\}$.
Recall the common notation $[R_i] = \{1,\dotsc,R_i\}$, so we may write $[R_i] \cup \{k\}$ for the set we want.
To simplify notation, we write $Z([R_i])$ for $Z_{R_i}^{R_i}$
and we write $Z([R_i] \cup \{k\})$ for the $(R_i+1) \times (R_i+1)$ skew-symmetric submatrix of $Z$
given by the rows (and columns) corresponding to the set $\{1,\dotsc,R_i,k\}$.
Since confusion seems unlikely we will drop the brackets and braces
and simply write $Z(R_i)$ and $Z(R_i \cup k)$.
Thus for example
\[
  Z(3\cup5) =
    \begin{pmatrix}
      0 & z_{1,2} & z_{1,3} & z_{1,5} \\
      -z_{1,2} & 0 & z_{2,3} & z_{2,5} \\
      -z_{1,3} & -z_{2,3} & 0 & z_{3,5} \\
      -z_{1,5} & -z_{2,5} & -z_{3,5} & 0
    \end{pmatrix} ,
\]
with rows and columns given by the set $3 \cup 5 = [3] \cup \{5\} = \{1,2,3,5\}$.
\begin{thm}\label{thm: pfaffians}
Let $B$ be a ring and $A = B[Z]$.
Let $J = J(Z,2t,R,r)$ be the ideal generated by $2t$-Pfaffians of $Z$
that have at least $r_i$ rows in $Z_{R_i}$ for $1 \leq i \leq p$.
For each $i$, if $r_i$ is even, let $J_i = P_{r_i}(Z(R_i))$,
and if $r_i$ is odd, let $J_i = \sum_{k=R_i+1}^n P_{r_i+1}(Z(R_i \cup k))$.
Then
\begin{equation}\label{eq: pfaffian decomposition}
  J = P_{2t}(Z) \cap J_1 \cap \dotsb \cap J_p .
\end{equation}
If $B$ is a domain then $P_{2t}(Z)$ is prime and each $J_i$ is a prime ideal.
\end{thm}

\begin{proof}
Each of the following is an order ideal in $\Pi(Z)$:
\begin{enumerate}
\item The set of $2t$-Pfaffians is the order ideal generated by $[n-2t+1,\dotsc,n]$.
\item The set of Pfaffians (of all sizes) with at least $r_i$ rows contained in $Z_{R_i}$.
If $r_i$ is even, this is the order ideal generated by $[R_i-r_i+1,\dotsc,R_i]$.
If $r_i$ is odd, this is the order ideal generated by $[R_i-r_i+1,\dotsc,R_i,n]$.
\end{enumerate}
So, by Lemma~\ref{lemma: ASL order ideal intersection}, $J$ is equal to the intersection
of the ideals $P_{2t}(Z)$ and, for each $i$, the ideal generated by the Pfaffians (of any size)
having at least $r_i$ rows in $Z_{R_i}$.

If $r_i$ is even then the ideal generated by Pfaffians
with at least $r_i$ rows in $Z_{R_i}$ is $P_{r_i}(Z(R_i))$.
Indeed, if $P$ is any Pfaffian with at least $r_i$ rows in $Z_{R_i}$ then
$P$ can be expanded as a combination of $r_i$-Pfaffians
involving those rows.

If $r_i$ is odd and $P$ is any Pfaffian with at least $r_i$ rows in $Z_{R_i}$,
then either $P$ actually has at least $r_i+1$ rows in $Z_{R_i}$
or else $P$ involves at least one more row, say the $k$th row, with $k > R_i$.
Either way, $P$ can be expanded as a combination of $(r_i+1)$-Pfaffians in $Z(R_i \cup k)$.
So $P$ lies in the sum given in the statement.
Conversely, every Pfaffian generator of the sum in the statement must have at least $r_i$ rows in $Z_{R_i}$.

Now suppose $B$ is a domain.
The set of $(\geq 2t)$-Pfaffians of $Z$ is cogenerated by $[1,\dotsc,2t-2]$.
This shows $P_{2t}(Z)$ is prime.
(Of course $P_{2t}(Z)$ is already well-known to be prime.)

To see that each $J_i$ is prime, note that the order ideal of Pfaffians generating $J_i$ is cogenerated by
either $m = [1,\dotsc,r_i-1,R_i+1,\dotsc,n]$ or $m' = [1,\dotsc,r_i-1,R_i+1,\dotsc,n-1]$,
whichever has even length (regardless of whether $r_i$ is even or odd).
Let us verify this.
For simplicity, suppose that $m$ has even length.
We must show that $\alpha \not \geq m$ if and only if $\alpha$ has at least $r_i$ rows in $Z_{R_i}$,
equivalently $\alpha \geq m$ if and only if $\alpha$ has $r_i-1$ or fewer rows in $Z_{R_i}$;
note that this is the criterion whether $r_i$ is even or odd.
Now $\alpha \geq m$ if and only if $|\alpha| \leq r_i-1$ or $|\alpha| \geq r_i$ and $\alpha_{r_i} \geq R_i + 1$.
The forward direction is obvious; conversely, under these conditions,
$\alpha_{r_i+t} \geq \alpha_{r_i} + t \geq R_i + 1 + t = m_{r_i + t}$ for all $0 \leq t \leq |\alpha|-r_i$,
so each entry of $\alpha$ is at least as great
as the corresponding entry of $m$; and in particular since every entry of $\alpha$ is at most $n$, $|\alpha| \leq |m|$.
This shows that $\alpha \geq m$.
So indeed $\alpha \geq m$ if and only if $\alpha$ has $r_i-1$ or fewer rows in $Z_{R_i}$.

The argument in case $|m'|$ is even is similar.
Note that $m'$ is as long as possible for a member of $\Pi(Z)$ with $R_i+1$ in the $r_i$ position;
so if $\alpha_{r_i} \geq R_i+1$ then $|\alpha| \leq |m'|$.
\end{proof}

Once again this writes $J$ as a possibly redundant intersection of prime ideals, if $B$ is a domain.
\begin{prop}
Suppose
\begin{enumerate}
\item $R_1 < \dotsb < R_p$,
\item $r_1 < \dotsb < r_p < 2t$,
\item $0 \leq r_i \leq R_i$ for each $i$,
\item $R_1 - r_1 < \dotsb < R_p - r_p < n-2t$.
\end{enumerate}
Then the intersection \eqref{eq: pfaffian decomposition} is irredundant.
\end{prop}
The proof is the same as before.

\begin{cor}
Let $Z = (z_{i,j})$ be a generic skew-symmetric matrix,
let $0 < p < q$,
and fix $A = B[Z]$ with degree $\deg z_{i,j} = 2p$ if $i,j \leq R$,
$\deg z_{i,j} = p+q$ if $i \leq R < j$,
and $\deg z_{i,j} = 2q$ if $i,j > R$.
Fix $t$ and $d$.
Let $r = \left\lceil \frac{2tq - d}{q-p} \right\rceil$.
Then $P_{2t}(Z)_{\leq d} = P_{2t}(Z) \cap I_r$
where $I_r$ is the ideal generated by Pfaffians with at least $r$ rows in $Z_R$.
If $r$ is even, $I_r = P_r(Z(R))$.
If $r$ is odd, $I_r = \sum_{k=R+1}^n P_{r+1}(Z(R \cup k))$
where $Z(R \cup k)$ is the $(R+1) \times R+1$ skew-symmetric submatrix of $Z$ given by the rows (and columns)
corresponding to the set $\{1,\dotsc,R,k\}$.
If $B$ is a domain then $P_{2t}(Z)$ and $I_r$ are prime.
\end{cor}
Indeed, a $2t$-Pfaffian $P$ with $r$ rows in $Z_R$ has degree $pr+q(2t-r)$;
the value of $r$ in the statement is the least integral solution to $\deg P \leq d$.

\bibliographystyle{amsplain}

\providecommand{\bysame}{\leavevmode\hbox to3em{\hrulefill}\thinspace}
\providecommand{\MR}{\relax\ifhmode\unskip\space\fi MR }
\providecommand{\MRhref}[2]{%
  \href{http://www.ams.org/mathscinet-getitem?mr=#1}{#2}
}
\providecommand{\href}[2]{#2}

\end{document}